\pdfoutput=1
\RequirePackage{ifpdf}
\ifpdf 
\documentclass[pdftex]{sigma}
\else
\documentclass{sigma}
\fi

\numberwithin{equation}{section}

\newtheorem{Theorem}{Theorem}[section]
\newtheorem{Lemma}[Theorem]{Lemma}
\newtheorem{Proposition}[Theorem]{Proposition}
{ \theoremstyle{definition}
\newtheorem{Remark}[Theorem]{Remark} }

\usepackage{stmaryrd}
\newcommand{\vol}{\operatorname{vol}}
\newcommand{\Ric} {\operatorname{Ric}}
\newcommand{\Hess} {\operatorname{Hess}}
\newcommand{\R}{\mathbb{R}}
\newcommand{ \tr } {\operatorname{tr}}
\newcommand{\Rm} {\operatorname{Rm}}
\newcommand{\scal} {\operatorname{scal}}
\newcommand{ \Sym } {\operatorname{Sym}}

\newcommand{ \floor }[1]{ \lfloor #1 \rfloor }

\begin{document}
\allowdisplaybreaks

\newcommand{\arXivNumber}{2005.02604}

\renewcommand{\thefootnote}{}

\renewcommand{\PaperNumber}{064}

\FirstPageHeading

\ShortArticleName{The Bochner Technique and Weighted Curvatures}

\ArticleName{The Bochner Technique and Weighted Curvatures\footnote{This paper is a~contribution to the Special Issue on Scalar and Ricci Curvature in honor of Misha Gromov on his 75th Birthday. The full collection is available at \href{https://www.emis.de/journals/SIGMA/Gromov.html}{https://www.emis.de/journals/SIGMA/Gromov.html}}}

\Author{Peter PETERSEN and Matthias WINK}

\AuthorNameForHeading{P.~Petersen and M.~Wink}

\Address{Department of Mathematics, University of California,\\ 520 Portola Plaza, Los Angeles, CA, 90095, USA}
\Email{\href{mailto:petersen@math.ucla.edu}{petersen@math.ucla.edu}, \href{mailto:wink@math.ucla.edu}{wink@math.ucla.edu}}

\ArticleDates{Received May 22, 2020, in final form June 29, 2020; Published online July 09, 2020}

\Abstract{In this note we study the Bochner formula on smooth metric measure spaces. We introduce weighted curvature conditions that imply vanishing of all Betti numbers.}


\Classification{53B20; 53C20; 53C21; 53C23; 58A14}

\renewcommand{\thefootnote}{\arabic{footnote}}
\setcounter{footnote}{0}

\vspace{-3mm}

\section{Introduction}

Let $(M,g)$ be an oriented Riemannian manifold, let $\vol_g$ denote its volume form and let $f$ be a smooth function on $M$. The triple $\big(M,g, e^{-f} \vol_g\big)$ is called a smooth metric measure space. Based on considerations from diffusion processes, Bakry--\'Emery \cite{BakryEmeryDiffusion} introduced the tensor \begin{gather*}
\Ric_f = \Ric + \Hess f
\end{gather*}
as a weighted Ricci curvature for a geometric measure space. In fact, this tensor appeared earlier in work of Lichnerowicz \cite{LichnerowiczBETensor}. Volume comparison theorems for smooth metric measure spaces with $\Ric_f$ bounded from below have been established by Qian~\cite{QianEstimatesWeightedVolume}, Lott \cite{LottGeometryBETensor}, Bakry--Qian \cite{BakryQianVolumeComparison} and Wei--Wylie \cite{WeiWylieComparosionGeoBE}.

In this note we study the Bochner technique on smooth metric measure spaces. The distortion of the volume element introduces a diffusion term to the Bochner formula
\begin{gather*}
\Delta_f \omega = ( d d^{*}_f + d^{*}_f d) \omega= \nabla^{*}_f \nabla \omega + \Ric(\omega) - ( \Hess f ) \omega,
\end{gather*}
where $\Ric$ is the Bochner operator on $p$-forms. Lott \cite{LottGeometryBETensor} proved that if $\Ric_f \geq 0$, then all $\Delta_f$-harmonic $1$-forms are parallel and, for compact manifolds, $H^1(M;\R)$ is isomorphic to the space of all parallel $1$-forms $\omega$ which satisfy $\big\langle \nabla e^{-f}, \omega \big\rangle = 0$. Moreover, if $\Ric_f > 0$, then all $\Delta_f$-harmonic $1$-forms vanish.

We introduce new weighted curvature conditions that imply rigidity and vanishing results for $\Delta_f$-harmonic $p$-forms for $p \geq 1$. We can restrict to $p$-forms $\omega$ for $1 \leq p \leq \big\lfloor \frac{n}{2}\big\rfloor$ since $\omega$ is parallel if and only if $\ast \omega$ is parallel, where $\ast$ denotes the Hodge star.

By convention, we will refer to the eigenvalues of the curvature operator simply as the eigenvalues of the associated curvature tensor.

\begin{theorem*}
Let $\big(M^n,g, e^{-f} \vol_g\big)$ be a smooth metric measure space. For $1 \leq p < \frac{n}{2}$ set
\begin{gather*}
h = \frac{1}{n-2p} \Hess f - \frac{\Delta f}{2(n-p)(n-2p)} g.
\end{gather*}
Let $\omega$ be a $\Delta_f$-harmonic $p$-form with $|\omega| \in L^2\big(M, e^{-f} \vol_g\big)$ for $1 \leq p < \frac{n}{2}$. Let $\lambda_1 \leq \dots \leq \lambda_{\genfrac(){0pt}{2}{n}{2}}$ denote the eigenvalues of the weighted curvature tensor $\Rm + h \owedge g$.

If $\lambda_1 + \dots + \lambda_{n-p} \geq 0$, then $\omega$ is parallel. If in addition $M$ is compact, then $H^{p}(M)= \big\lbrace \omega \in \Omega^p(M) \, \vert \, \nabla \omega =0 \ \text{and} \ i_{\nabla f} \omega = 0 \big\rbrace$.

If $\lambda_1 + \dots + \lambda_{n-p} > 0$, then $\omega$ vanishes. If in addition $M$ is compact, then the Betti numbers $b_p(M)$ and $b_{n-p}(M)$ vanish for $1 \leq p < \frac{n}{2}$.
\end{theorem*}

For $p=1$ the Ricci curvature of the modified curvature tensor is the Bakry--\'Emery Ricci tensor, and the assumption in the Theorem implies that it is nonnegative. In this sense the Theorem is a generalization of Lott's \cite{LottGeometryBETensor} results for $1$-forms.

A stronger curvature assumption also allows control in the middle dimension $p = \frac{n}{2}$. Recall that a curvature tensor is $l$-nonnegative (positive) if the sum of its lowest $l$ eigenvalues is nonnegative (positive).

\begin{proposition*}
Let $\big(M^n,g, e^{-f} \vol_g\big)$ be a smooth metric measure space. Let $\mu_1 \leq \dots \leq \mu_n$ denote the eigenvalues of $\Hess f$ and let $1 \leq p \leq \big\lfloor  \frac{n}{2}\big\rfloor$.

Let $\omega$ be a $\Delta_f$-harmonic $p$-form with $|\omega| \in L^2\big(M, e^{-f} \vol_g\big)$. If the weighted curvature tensor
\begin{gather*}
\Rm + \frac{\sum\limits_{i=1}^p \mu_i}{2p(n-p)} g \owedge g
\end{gather*}
is $(n-p)$-nonnegative, then $\omega$ is parallel. If it is $(n-p)$-positive, then $\omega$ vanishes.

In particular, if $M$ is compact, then $H^{p}(M)= \big\lbrace \omega \in \Omega^p(M) \, \vert \, \nabla \omega =0 \ \text{and} \ i_{\nabla f} \omega = 0 \big\rbrace$ and in case the weighted curvature tensor is $(n-p)$-positive, the Betti numbers $b_p(M)$ and $b_{n-p}(M)$ vanish.
\end{proposition*}

The notation in this paper builds up on the presentation in \cite[Chapter~9]{PetersenRiemGeom} and \cite{PetersenWinkBochner}.

\section{Preliminaries}

\subsection{Algebraic curvature tensors}
For an $n$-dimensional Euclidean vector space $(V,g)$ let $\mathcal{T}^{(0,k)}(V)$ denote the vector space of $(0,k)$-tensors and $\operatorname{Sym}^2(V)$ the vector space of symmetric $(0,2)$-tensors on $V$.

Let $\mathcal{C}(V)$ denote the vector space of $(0,4)$-tensors with $T(X,Y,Z,W) = - T(Y,X,Z,W) = T(Z,W,X,Y)$. If $T$ also satisfies the algebraic Bianchi identity, then $T$ is called algebraic curvature tensor, $T \in \mathcal{C}_B(V)$.

The Kulkarni--Nomizu product of $S_1, S_2 \in \operatorname{Sym}^2(V)$ is given by
\begin{gather*}
(S_1 \owedge S_2)(X,Y,Z,W) =   S_1(X,Z)S_2(Y,W)-S_1(X,W)S_2(Y,Z) \\
\hphantom{(S_1 \owedge S_2)(X,Y,Z,W) =}{}  +S_1(Y,W)S_2(X,Z)-S_1(Y,Z)S_2(X,W).
\end{gather*}
With this convention the algebraic curvature tensor $I= \frac{1}{2} g \owedge g$ corresponds to the curvature tensor of the unit sphere.

Recall that the decomposition of $\mathcal{C}(V)$ into $O(n)$-irreducible components is given by
\begin{gather*}
\mathcal{C}(V) = \langle I \rangle \oplus \langle \mathring{\Ric} \rangle \oplus \langle W \rangle \oplus \Lambda^4 V,
\end{gather*}
where $\langle \mathring{\Ric} \rangle = S_0^2(V) \owedge g$ is the subspace of algebraic curvature tensors of trace-free Ricci type, $S_0^2(V)= \big\lbrace h \in \operatorname{Sym}^2(V) \, \vert \, \tr(h)= 0 \big\rbrace$, and $\langle W \rangle$ denotes the subspace of Weyl tensors.

Explicitly, every algebraic curvature tensor decomposes as
\begin{gather*}
\Rm = \frac{\scal}{2(n-1)n} g \owedge g + \frac{1}{n-2} \mathring{\Ric} \owedge g + W.
\end{gather*}

\subsection{Lichnerowicz Laplacians on smooth metric measure spaces}
Let $(M,g,f)$ be a smooth metric measure space. The formal adjoints of the exterior and covariant derivative with respect to the measure $e^{-f} \vol_g$ are given by
\begin{gather*}
d^{*}_f = d^{*} + i_{\nabla f} \qquad \text{and} \qquad \nabla^{*}_f = \nabla^{*} + i_{\nabla f}.
\end{gather*}
More generally, for a vector field $U$ on $M$, we will consider
\begin{gather*}
d^{*}_U = d^{*} + i_{U} \qquad \text{and} \qquad \nabla^{*}_U = \nabla^{*} + i_{U}.
\end{gather*}

The associated generalized Lichnerowicz Laplacian on $(0,k)$-tensors is given by
\begin{gather*}
\Delta_U T = \nabla^{*}_U \nabla T + \Ric(T) - (\nabla U) T,
\end{gather*}
where the curvature term is given by
\begin{gather*}
\Ric(T)(X_1, \dots, X_k) = \sum_{i=1}^k \sum_{j=1}^n (R(X_i,e_j)T) (X_1, \dots, e_j, \dots, X_k).
\end{gather*}
A tensor $T$ is called {\em $U$-harmonic} if $\Delta_U T =0$.

To emphasize that the curvature term is calculated with respect to the curvature tensor $\Rm$, we will also write $\Ric_{\Rm}(T)$ for $\Ric(T)$.

Recall that for an endomorphism $L$ of $V$ and a $(0,k)$-tensor $T$ we have
\begin{gather*}
(LT)(X_1, \dots, X_k) = - \sum_{i=1}^k T(X_1, \dots,L(X_i), \dots, X_k).
\end{gather*}
In particular, the Ricci identity implies that the definition of the curvature term in the Lichnerowicz Laplacian naturally carries over to algebraic curvature tensors.

\begin{Proposition}\label{BochnerFormulasExample}
Let $(M,g)$ be a Riemannian manifold and $U$ a vector field on~$M$. For a~$(0,k)$-tensor $T$ on $M$ set $\Ric_U(T)=\Ric(T) - (\nabla U)T$.
\begin{enumerate}\itemsep=0pt
\item[$(a)$] Every $p$-form satisfies
\begin{gather*}
( d d^{*}_U + d^{*}_U d ) \omega = \nabla^{*}_U \nabla \omega + \Ric_U(\omega).
\end{gather*}
\item[$(b)$] Every symmetric $(0,2)$-tensor satisfies
\begin{gather*}
 ( \nabla_X \nabla^{*}_U T  )(X) + \big( \nabla^{*}_U d^{\nabla} T \big) (X,X) =  ( \nabla^{*}_U \nabla T  )(X,X)+ \frac{1}{2}  ( \Ric_U T  ) (X,X),
\end{gather*}
where $d^{\nabla}T(Z,X,Y)=  (\nabla_X T  )(Y,Z) - (\nabla_Y T )(X,Z)$.
\end{enumerate}
\end{Proposition}
\begin{proof}
(a) The case $U=0$ recovers the well-known Bochner formula. The generalized Hodge Laplacian satisfies
\begin{gather*}
d d^{*}_U + d^{*}_U d = d d^{*} + d^{*} d + d i_U + i_U d = \Delta + L_U.
\end{gather*}
In addition to the classical Lichnerowicz Laplacian we have on the right hand side
\begin{gather*}
\nabla_U - ( \nabla U ) = L_U
\end{gather*}
and thus all diffusion terms balance out.

(b) As in (a), it suffices to consider all terms that depend on $U$ and show that
\begin{gather*}
( \nabla_X i_U h ) (X) + \big( i_U d^{\nabla} h \big) (X,X) =  ( \nabla_U h ) (X,X) - \frac{1}{2} ( (\nabla U) h)(X,X).
\end{gather*}
This is a straightforward calculation
\begin{gather*}
 ( \nabla_X i_U h  ) (X)   + \big( i_U d^{\nabla} h \big) (X,X) \\
 \qquad{} =  ( \nabla_X h ) (U,X) + h ( \nabla_X U, X ) +  ( \nabla_U h ) (X,X) -  ( \nabla_ X h )  ( U, X  ) \\
\qquad{} =  ( \nabla_U h  ) (X,X) + h  ( \nabla_X U, X ) \\
\qquad{} =  ( \nabla_U h ) (X,X) - \frac{1}{2}  ( ( \nabla U ) h ) (X,X).\tag*{\qed}
\end{gather*}\renewcommand{\qed}{}
\end{proof}

\begin{Remark}The curvature tensor $\Rm$ of a Riemannian manifold satisfies
\begin{gather*}
\nabla^{*}_U \nabla \Rm + \frac{1}{2} \Ric_U( \Rm ) =   \frac{1}{2}  ( \nabla_X \nabla_U^{*} \Rm )(Y,Z,W) - \frac{1}{2}  ( \nabla_Y \nabla_U^{*} \Rm  ) (X,Z,W) \\
\hphantom{\nabla^{*}_U \nabla \Rm + \frac{1}{2} \Ric_U( \Rm ) =}{} + \frac{1}{2}  ( \nabla_Z \nabla_U^{*} \Rm  )(W, X,Y) - \frac{1}{2}  ( \nabla_W \nabla_U^{*} \Rm  ) (Z,X,Y).
\end{gather*}
A straightforward computation based on the second Bianchi identity shows that all terms that involve~$U$ cancel.
\end{Remark}

The Bochner technique with diffusion relies on the following basic observations. Firstly, the maximum principle implies:

\begin{Lemma}\label{BochnerTechniqueWithDiffusion}
Let $(M,g)$ be a Riemannian manifold, $U$ a vector field on~$M$. Let~$T$ be a tensor such that
\begin{gather*}
g ( \nabla_U^{*} \nabla T, T) \leq 0.
\end{gather*}
If $| T |$ has a maximum, then $T$ is parallel.
\end{Lemma}

\begin{Remark}\label{IsomorphismDeRhamColomology}
Note that a $p$-form $\omega$ satisfies $(dd^{*}_U + d^{*}_Ud ) \omega = 0$ if and only if $d \omega =0$ and $d^{*}_U \omega = 0$.

As in \cite{LottGeometryBETensor}, if $M$ is compact and oriented, standard elliptic theory implies that
\begin{gather*}
H^p(M) = \big\lbrace \omega \in \Omega^p(M) \, \vert \, d \omega =0 \ \text{and} \ d^{*}_U \omega = 0 \big\rbrace.
\end{gather*}
Suppose that $\Ric_U \geq 0$ on $p$-forms. It follows that a $p$-form $\omega$ is $U$-harmonic if and only if $\omega$ is parallel and $i_U \omega = 0$. Thus,
\begin{gather*}
H^{p}(M)= \big\lbrace \omega \in \Omega^p(M) \, \vert \, \nabla \omega =0 \ \text{and} \ i_{U} \omega = 0 \big\rbrace.
\end{gather*}
\end{Remark}

If $U=\nabla f$, then we can use integration to conclude:

\begin{Lemma}\label{BochnerTechniqueSMMS}
Let $(M,g,f)$ be a smooth metric measure space with $\int_M e^{-f} \vol_g < \infty$. If $T$ is a~$(0,k)$-tensor with $|T| \in L^{2}\big(M, e^{-f} \vol_g\big)$ and
\begin{gather*}
g ( \nabla_f^{*} \nabla T, T) \leq 0,
\end{gather*}
then $T$ is parallel.
\end{Lemma}

\section{Weighted Lichnerowicz Laplacians}

The idea of this section is to define a weighted curvature tensor $\widetilde{\Rm}$ so that for a given symmetric tensor $S$ the curvature term of the Lichnerowicz Laplacian satisfies
\begin{gather*}
g( \Ric_{\Rm}(T) - (S)T, T ) = g\big( \Ric_{\widetilde{\Rm}}(T),T\big).
\end{gather*}

This will be achieved by adding a weight to the Ricci tensor of $\Rm$, leaving the Weyl curvature unchanged. The specific weight will depend on the irreducible components of the tensors of type~$T$, e.g., it is different for forms and symmetric tensors.

Let $T$ be a $(0,k)$-tensor. For $\tau_{ij} \in S_k$ let $T \circ \tau_{ij}$ denote the transposition of the $i$-th and $j$-th entries of $T$ and for $h \in \Sym^2(V)$ let $c_{ij}(h \otimes T)$ denote the contraction of $h$ with the $i$-th and $j$-th entries of $T$.

\begin{Proposition}\label{GeneralFormulaBochnerOnHWedgeG}
For $h \in \Sym^2(V)$ let $H \colon V \to V$ denote the associated symmetric operator. If $T \in \mathcal{T}^{(0,k)}(V)$, then
\begin{gather*}
\Ric_{h \owedge g}(T) (X_1, \dots, X_k)
=   2 \sum_{i \neq j} ( T \circ \tau_{ij}) (X_1, \dots, H(X_i), \dots, X_k) \\
\hphantom{\Ric_{h \owedge g}(T) (X_1, \dots, X_k)=}{}
 - \sum_{i \neq j} g(X_i, X_j) c_{ij}(h \otimes T)( X_1, \dots, \widehat{X}_i, \dots, \widehat{X}_j, \dots, X_k ) \\
\hphantom{\Ric_{h \owedge g}(T) (X_1, \dots, X_k)=}{} - \sum_{i \neq j} h(X_i, X_j) c_{ij}(g \otimes T)( X_1, \dots, \widehat{X}_i, \dots, \widehat{X}_j, \dots, X_k ) \\
\hphantom{\Ric_{h \owedge g}(T) (X_1, \dots, X_k)=}{} -(n-2) (HT)(X_1, \dots, X_k) + k \cdot \tr(h) T( X_1, \dots, X_k).
\end{gather*}
\end{Proposition}
\begin{proof}The algebraic curvature tensor $R = h \owedge g$ satisfies
\begin{gather*}
R(X,Y,Z,W) =  g(H(X),Z)g(Y,W)-g(Y,Z)g(H(X),W) \\
\hphantom{R(X,Y,Z,W) =}{} +g(X,Z)g(H(Y),W)-g(H(Y),Z)g(X,W)
\end{gather*}
and hence
\begin{gather*}
R(X,Y)Z =  ( H(X) \wedge Y + X \wedge H(Y)  ) Z
\end{gather*}
is the corresponding $(1,3)$-tensor. It follows that
\begin{gather*}
\Ric_{h \owedge g}(T)(X_1, \dots, X_k) =   \sum_{i=1}^k \sum_{a=1}^n  ( R(X_i, e_a)T  ) ( X_1, \dots, e_a, \dots, X_k) \\
\hphantom{\Ric_{h \owedge g}(T)(X_1, \dots, X_k)}{} =   \sum_{i=1}^k \sum_{a=1}^n (( H(X_i) \wedge e_a )T)( X_1, \dots, e_a, \dots, X_k) \\
\hphantom{\Ric_{h \owedge g}(T)(X_1, \dots, X_k)=}{}  + \sum_{i=1}^k \sum_{a=1}^n (( X_i \wedge H(e_a) )T)( X_1, \dots, e_a, \dots, X_k).
\end{gather*}

It is straightforward to calculate
\begin{gather*}
\sum_{i=1}^k \sum_{a=1}^n  ((X_i \wedge H(e_a)) T)(X_1, \dots, e_a, \dots, X_k) \\
\qquad{} =  \sum_{i \neq j} \sum_{a=1}^n T( X_1, \dots, (H(e_a) \wedge X_i) X_j, \dots, e_a, \dots, X_k) \\
 \qquad\quad{} + \sum_{i=1}^k \sum_{a=1}^n T(X_1, \dots, (H(e_a) \wedge X_i)e_a, \dots, X_k) \\
\qquad {}=  \sum_{i \neq j} \sum_{a=1}^n T( X_1, \dots, g(H(e_a), X_j) X_i - g(X_i, X_j)H(e_a), \dots, e_a, \dots, X_k) \\
\qquad\quad{}+ \sum_{i=1}^k \sum_{a=1}^n T(X_1, \dots, g(H(e_a),e_a) X_i - g(e_a, X_i) H(e_a), \dots, X_k) \\
\qquad {}=  \sum_{i \neq j} \sum_{a=1}^n T( X_1, \dots, g(e_a, H(X_j)) X_i, \dots, e_a, \dots, X_k) \\
\qquad\quad{} - \sum_{i \neq j} \sum_{a=1}^n g(X_i, X_j) T( X_1, \dots, H(e_a), \dots, e_a, \dots, X_k) \\
 \qquad\quad{}+ \sum_{i=1}^k \sum_{a=1}^n h(e_a,e_a) T(X_1, \dots, X_k)
 - \sum_{i=1}^k \sum_{a=1}^n T(X_1, \dots, H  ( g(e_a, X_i)e_a  ), \dots, X_k) \\
\qquad {}=  \sum_{i \neq j} T ( X_1, \dots, X_i, \dots, H(X_j), \dots, X_k) \text{ [here } X_i \text{ is in the j-th position]} \\
\qquad\quad{}- \sum_{i \neq j} \sum_{a,b=1}^n  g(X_i, X_j) h(e_a, e_b) T( X_1, \dots, e_b, \dots, e_a, \dots, X_k)
 + k \cdot \tr( h ) T(X_1, \dots, X_k) \\
\qquad\quad{} - \sum_{i=1}^k T(X_1, \dots, H(X_i), \dots, X_k) \\
\qquad {}=  \sum_{i \neq j} (T \circ \tau_{ij}) ( X_1, \dots, H(X_j), \dots, X_i, \dots, X_k) \text{ [here } H(X_j) \text{ is in the j-th position]} \\
\qquad\quad{} - \sum_{i \neq j} g(X_i, X_j) c_{ij} (h \otimes T)( X_1, \dots, \widehat{X_i}, \dots, \widehat{X_j}, \dots, X_k) \\
 \qquad\quad{} + k \cdot \tr( h ) T(X_1, \dots, X_k)  + (HT)(X_1, \dots, X_k).
\end{gather*}

Similarly one computes
\begin{gather*}
\sum_{i=1}^k \sum_{a=1}^n   ((H(X_i) \wedge e_a) T)(X_1, \dots, e_a, \dots, X_k) \\
\qquad{} =   \sum_{i \neq j} (T \circ \tau_{ij}) ( X_1, \dots, X_j, \dots, H(X_i), \dots, X_k) \text{ [here } X_j \text{ is in the j-th position]} \\
\qquad\quad{}  -\! \sum_{i \neq j} h(X_i, X_j) c_{ij} (g \otimes T)\big(X_1, {\dots}, \widehat{X_i}, {\dots}, \widehat{X_j}, {\dots}, X_k\big)   - (n-1) (HT) (X_1, {\dots}, X_k) .
\end{gather*}

Adding up both terms yields $\Ric_{h \owedge g}(T)$ as claimed.
\end{proof}

\begin{Proposition}\label{WeightedCurvatureTerm}
Let $(V,g)$ be an $n$-dimensional Euclidean vector space and $h \in \Sym^2(V)$. The following hold:
\begin{enumerate}\itemsep=0pt
\item[$1.$] Every $T \in \operatorname{Sym}^2(V)$ satisfies
\begin{gather*}
\Ric_{h \owedge g}(T)   = - n HT - 2 \langle T,h \rangle g-2 \tr(T)h + 2 \tr(h) T , \\
g(\Ric_{h \owedge g}(T),T)   = - n g(HT,T) -4 \tr(T) \langle T, h \rangle + 2 \tr(h) |T|^2.
\end{gather*}
\item[$2.$] Every $p$-form $\omega$ satisfies
\begin{gather*}
\Ric_{h \owedge g}(\omega)   = - (n-2p) H \omega + p \tr(h) \omega, \\
g(\Ric_{h \owedge g}(\omega), \omega)   = - (n-2p) g(H \omega, \omega) + p \tr(h) |\omega|^2.
\end{gather*}
\item[$3.$] Every algebraic $(0,4)$-curvature tensor $\Rm$ satisfies
\begin{gather*}
\Ric_{h \owedge g}(\Rm)  = -2  ( h \owedge \Ric  ) -2 g \owedge ( c_{24} ( h \otimes \Rm  ) ) -(n-2) H \Rm + 4 \tr(h) \Rm.
\end{gather*}
\end{enumerate}
\end{Proposition}
\begin{proof}
(a) Due to the symmetry of $T$ it follows that
\begin{gather*}
\Ric_{h \owedge g}(T)(X_1,X_2) =   2 \lbrace T(H(X_1), X_2) + T(X_1, H(X_2) \rbrace \\
\hphantom{\Ric_{h \owedge g}(T)(X_1,X_2) =}{} - 2 \lbrace g(X_1, X_2) \langle h, T \rangle + h(X_1, X_2) \tr(T) \rbrace \\
\hphantom{\Ric_{h \owedge g}(T)(X_1,X_2) =}{} - (n-2) (HT)(X_1, X_2) + 2 \tr(h) T(X_1, X_2).
\end{gather*}

(b) Since $\omega \circ \tau_{ij} = - \omega$ for every transposition $\tau_{ij}$ it follows that
\begin{align*}
\sum_{i \neq j} (\omega \circ \tau_{ij})(X_1, \dots, H(X_i), \dots, X_p)& =
  - \sum_{i \neq j} \omega(X_1, \dots, H(X_i), \dots, X_p) \\
& =   - (p-1) \sum_{i=1}^p \omega(X_1, \dots, H(X_i), \dots, X_p) \\
& =   (p-1) (H \omega)(X_1, \dots, X_p)
\end{align*}
and furthermore $c_{ij}(g \otimes \omega) = c_{ij}(h \otimes \omega)=0$ for all $i \neq j$. This implies the claim.

(c) The symmetries of the curvature tensor imply that
\begin{gather*}
  \sum_{i \neq j} ( \Rm \circ \tau_{ij} )(X_1, \dots, H(X_i), \dots, X_4)   \\
\!\!\qquad{}= (H \Rm)(X_1, X_2, X_3, X_4) + (H \Rm)( X_2, X_3, X_1, X_4) + (H \Rm)( X_3, X_1, X_2, X_4)=0
\end{gather*}
due to the first Bianchi identity.

Computing with respect to an orthonormal eigenbasis of $H$ it follows that
\begin{gather*}
( g( \cdot, \cdot ) c_{12}(h \otimes \Rm)) (X, Y, Z, W) =   0, \\
( g( \cdot, \cdot ) c_{13}(h \otimes \Rm)) (X, Y, Z, W) =
  \sum_{a,b=1}^n g(X,Z) \Rm(g(H(e_a),e_b) e_b, Y, e_a, W) \\
\hphantom{( g( \cdot, \cdot ) c_{13}(h \otimes \Rm)) (X, Y, Z, W)}{} =   \sum_{a=1}^n g(X,Z) \Rm(H(e_a), Y, e_a, W) \\
\hphantom{( g( \cdot, \cdot ) c_{13}(h \otimes \Rm)) (X, Y, Z, W)}{} =   \sum_{a=1}^n g(Z,X) \Rm(e_a, Y, H(e_a), W) \\
\hphantom{( g( \cdot, \cdot ) c_{13}(h \otimes \Rm)) (X, Y, Z, W)}{} =   ( g( \cdot, \cdot ) c_{31}(h \otimes \Rm)) (X, Y, Z, W).
\end{gather*}
This implies
\begin{gather*}
\sum_{i \neq j} ( g( \cdot, \cdot )   c_{ij}(h \otimes \Rm)) (X, Y, Z, W) \\
\qquad{}=   2 \sum_{i=1}^n \lbrace g(X,Z) \Rm(H(e_i), Y, e_i, W) + g(X,W) \Rm(H(e_i), Y, Z, e_i) \\
\qquad\quad{} + g(Y,Z) \Rm(X, H(e_i), e_i, W ) + g(Y,W) \Rm(X, H(e_i), Z, e_i) \rbrace \\
\qquad {}=   2 \sum_{i=1}^n \lbrace g(X,Z) \Rm(Y, H(e_i), W, e_i) - g(X,W) \Rm(Y, H(e_i), Z, e_i) \\
\qquad\quad{}- g(Y,Z) \Rm( X, H(e_i), W, e_i) + g(Y,W) \Rm(X, H(e_i), Z, e_i) \rbrace \\
\qquad {}=   2 \left( g \owedge \left[ \sum_{i=1}^n \Rm( \cdot, H(e_i), \cdot, e_i) \right] \right) (X, Y, Z,W) \\
\qquad {} =   2 \left( g \owedge c_{24} ( h \otimes \Rm) \right) (X, Y, Z,W).
\end{gather*}
Similarly it follows that
\begin{gather*}
\sum_{i \neq j} ( h( \cdot, \cdot ) c_{ij}(g \otimes \Rm))
= 2 \left( h \owedge c_{24} ( g \otimes \Rm) \right)
= 2 \left( h \owedge \Ric \right).
\end{gather*}
This completes the proof.
\end{proof}

\begin{Remark} For a Weyl tensor $W$ and $h$ a symmetric $(0,2)$-tensor it is not hard to check that $\Ric_{h \owedge g}(W)$ satisfies
\begin{gather*}
g( \Ric_{h \owedge g}(W),W)   = - (n-2) g( HW, W) + 4 \tr(h) |W|^2, \\
g\big( \Ric_{h \owedge g}(W), g \owedge \mathring{\Ric}\big)
  = - 8(n-2)  \langle c_{24}(h \otimes W), \Ric \rangle
  = - 8(n-2) \big\langle c_{24}(\mathring{h} \otimes W), \mathring{\Ric} \big\rangle, \\
g( \Ric_{h \owedge g}(W), g \owedge g)   = 0.
\end{gather*}
It is worth noting that there are trace-free symmetric $(0,2)$-tensors $h_1$, $h_2$ such that the curvature tensor $h_1 \owedge h_2$ is Weyl.
\end{Remark}

The main Theorem follows as in Proposition~\ref{BochnerWithDiffusionForForms} below by using Lemma~\ref{BochnerTechniqueSMMS} instead of Lemma~\ref{BochnerTechniqueWithDiffusion}. The description of the de Rham cohomology groups follows from Remark~\ref{IsomorphismDeRhamColomology}.

\begin{Proposition}\label{BochnerWithDiffusionForForms}
Let $(M,g)$ be a Riemannian manifold and let $U$ be a vector field on $M$. Set $S= \nabla U$ and for $1 \leq p < \frac{n}{2}$ set
\begin{gather*}
H = \frac{1}{n-2p} S - \frac{1}{2(n-p)(n-2p)} \tr(S) I,
\end{gather*}
where $I \colon TM \to TM$ denotes the identity operator.

Suppose that the eigenvalues $\lambda_1 \leq \dots \leq \lambda_{\genfrac(){0pt}{2}{n}{2}}$ of the weighted curvature tensor $\Rm + h \owedge g$ satisfy
\begin{gather*}
\lambda_1 + \dots + \lambda_{n-p} \geq 0
\end{gather*}
and let $\omega$ be a $U$-harmonic $p$-form for $1 \leq p < \frac{n}{2}$.

If $| \omega |$ achieves a maximum, then $\omega$ is parallel. If in addition the inequality is strict, then $\omega$ vanishes.
\end{Proposition}
\begin{proof}
Proposition \ref{WeightedCurvatureTerm} (b) and $- I \omega = p \omega$ imply that
\begin{align*}
g( \Ric_{h \owedge g} \omega, \omega ) & = - (n-2p) g( H \omega, \omega) + p \tr(h) | \omega|^2
  = - g( ( (n-2p)H+ \tr(h) I ) \omega, \omega ) \\
& = - g \left( \left( S- \frac{\tr(S)}{2(n-p)} I + \frac{\tr(S)}{2(n-p)} I \right) \omega, \omega \right)
  = - g( S \omega, \omega ).
\end{align*}
Thus the Bochner formula takes the form
\begin{gather*}
\Delta_U \omega = \nabla^{*}_U \nabla \omega + \Ric(\omega) - (\nabla U) \omega = \nabla^{*}_U \nabla \omega + \Ric_{\Rm + h \owedge g}(\omega).
\end{gather*}
The argument in \cite[proof of Theorem A]{PetersenWinkBochner} shows that $\Ric_{\Rm+h \owedge g}( \omega ) \geq 0$. Lemma \ref{BochnerTechniqueWithDiffusion} implies the claim.

If the inequality is strict, then the same argument shows that $\Ric_{\Rm+h \owedge g}( \omega ) > 0$ unless $\omega =0$.
\end{proof}

The above approach only works for $p = \frac{n}{2}$ if $S$ is a multiple of the identity. However, we have

\begin{Proposition}
Let $(M,g)$ be an $n$-dimensional Riemannian manifold and let $U$ be a vector field on $M$. Set $S= \nabla U$ and fix $1 \leq p \leq \big\lfloor \frac{n}{2}\big\rfloor$. Let $\mu_1 \leq \dots \leq \mu_n$ denote the eigenvalues of~$S$. Suppose that the weighted curvature tensor
\begin{gather*}
\Rm + \frac{\sum\limits_{i=1}^p \mu_i}{2p(n-p)} g \owedge g
\end{gather*}
is $(n-p)$-nonnegative. If $\omega$ is a $U$-harmonic $p$-form $\omega$ such that $|\omega|$ has a maximum, then $\omega$ is parallel. If in addition the weighted curvature tensor is $(n-p)$-positive, then $\omega$ vanishes.
\end{Proposition}
\begin{proof}Calculating with respect to an orthonormal eigenbasis for $S$ it follows that
\begin{gather*}
- g( (S \omega), \omega) = - \sum_{i_1 < \dots < i_p} (S \omega )_{i_1 \dots i_p} \omega_{i_1 \dots i_p}
 = \sum_{i_1 < \dots < i_p} \left( \sum_{j=1}^p \mu_{i_j} \right) ( \omega_{i_1 \dots i_p } )^2
 \geq \left( \sum_{i=1}^p \mu_{i} \right) | \omega |^2.
\end{gather*}
Let $\lbrace \lambda_{\alpha} \rbrace$ denote the eigenvalues of
(the curvature operator associated to)
$\Rm$ and let $\lbrace \Xi_{\alpha} \rbrace$ be an orthonormal eigenbasis. It follows from \cite[Proposition 1.6]{PetersenWinkBochner} that
\begin{align*}
g( \Ric_{\Rm}( \omega), \omega ) - g( S \omega, \omega)  \geq \sum_{\alpha} \lambda_{\alpha} | \Xi_{\alpha} \omega |^2 + \left( \sum_{i=1}^p \mu_i \!\right) | \omega |^2
  = \sum_{\alpha} \!\left(\! \lambda_{\alpha} + \frac{\sum\limits_{i=1}^p \mu_i}{p(n-p)}\! \right) | \Xi_{\alpha} \omega |^2.
\end{align*}
The proof can now be completed as in Proposition~\ref{BochnerWithDiffusionForForms}.
\end{proof}

This principle can also be applied to $(0,2)$-tensors.

\begin{Proposition}Let $T \in \Sym^2(V)$ with $\tr(T)=0$, let $S= \nabla U$ and set
\begin{gather*}
H=\frac{S}{n} - \frac{\tr(S)}{2n^2} I.
\end{gather*}

Let $\lambda_1 \leq \dots \leq \lambda_{\genfrac(){0pt}{2}{n}{2}}$ denote the eigenvalues of the weighted curvature tensor $\Rm + h \owedge g$ and suppose that
\begin{gather*}
\lambda_{1} + \dots + \lambda_{\floor{\frac{n}{2}}} \geq 0.
\end{gather*}
If $T$ is $U$-harmonic and $|T|$ has a maximum, then~$T$ is parallel. If in addition the inequality is strict, then $T$ vanishes.
\end{Proposition}
\begin{proof}Proposition \ref{WeightedCurvatureTerm}(a) implies that
\begin{align*}
g  ( \Ric_{h \owedge g}(T), T ) & = - n g \left( \left(H + \frac{\tr(h)}{n} I \right)T ,T \right) \\
&  = - n g \left( \left( \frac{S}{n} - \frac{\tr(S)}{2n^2} I + \frac{\tr(S)}{2n^2} I \right) T, T \right)
  = - g( ST,T ).
\end{align*}
It follows from Proposition \ref{BochnerFormulasExample}(b) that
\begin{gather*}
\left( \nabla_X \nabla^{*}_U T \right)(X) + \left( \nabla^{*}_U d^{\nabla} T \right) (X,X) = \left( \nabla^{*}_U \nabla T \right)(X,X)+ \frac{1}{2} \left( \Ric_{\Rm+h \owedge g} T \right) (X,X).
\end{gather*}
As in \cite[Lemma 2.1 and Proposition~2.9]{PetersenWinkBochner} we conclude that $\Ric_{\Rm + h \owedge g}(T) \geq 0$. When the inequality is strict, the argument shows moreover $\Ric_{\Rm + h \owedge g}(T) > 0$ unless $T = 0$. This uses again that $T$ is trace-less.

An application of Lemma \ref{BochnerTechniqueSMMS} as before implies the claim.
\end{proof}

\subsection*{Acknowledgements}

We would like to thank the referees for useful comments.

\pdfbookmark[1]{References}{ref}
\LastPageEnding

\end{document}